\documentclass[12pt]{article}

\usepackage{amsmath, amsthm, amssymb}

\newtheorem{Theorem}{Theorem}[section]

\newtheorem{Lemma}[Theorem]{Lemma}
\newtheorem{Corollary}[Theorem]{Corollary}

\theoremstyle{definition}
\newtheorem{Definition}[Theorem]{Definition}

\newtheorem*{Problem}{Problem}

\def\Longmapsto{\DOTSB\mapstochar\Longrightarrow}

\begin{document}

\title{Undecidability of the problem of recognizing axiomatizations for implicative propositional calculi}

\author{Grigoriy V. Bokov\\
        \small Department of Mathematical Theory of Intelligent Systems\\
        \small Lomonosov Moscow State University\\
        \small Moscow, Russian Federation\\
        \small E-mail: bokovgrigoriy@gmail.com}

\maketitle

\begin{abstract}
In this paper we consider propositional calculi, which are finitely axiomatizable extensions of intuitionistic implicational propositional calculus together with the rules of modus ponens and substitution. We give a proof of undecidability of the following problem for these calculi: whether a given finite set of propositional formulas constitutes an adequate axiom system for a fixed propositional calculus. Moreover, we prove the same for the following restriction of this problem: whether a given finite set of theorems of a fixed propositional calculus derives all theorems of this calculus. The proof of these results is based on a reduction of the undecidable halting problem for the tag systems introduced by Post.
\end{abstract}


\section{Introduction}

In general, a propositional calculus is given by a finite set of propositional formulas over some signature together with a finite set of rules of inferences. The problem of recognizing axiomatizations for a propositional calculus is formulated as follows: whether a given finite set of propositional formulas constitutes (axiomatizes) an adequate axiom system for this calculus, i.e., each formula of the calculus is derivable from a given set of formulas by the rules of the calculus. The question of decidability of this problem was proposed by Tarski in 1946~\cite{Sinaceur:2000}. In this paper we consider only the propositional calculus with the rules of modus ponens and substitution.

The undecidability of recognizing axiomatizations for the classical propositional calculus was obtained due to Linial and Post in 1949~\cite{LinialPost:49}. They gave sketch of proofs for a number of results, one of them expressible in the form that it is undecidable whether a given finite set of propositional formulas axiomatizes all classical tautologies. Note that they considered only formulas over the signature $\{ \neg, \vee \}$ and the rule of modus ponens was formulated appropriately. Later the proof of their result was restored by Davis~\cite[pp.~137--142]{Davis:58}, and a complete proof appeared in the work of Yntema~\cite{Yntema:64}.

For the intuitionistic propositional calculus over the signature $\left\{ \neg, \vee, \&, \to \right\}$ the same result was proved by Kuznetsov in 1963~\cite{Kuznetsov:63}. Moreover, he proved that this holds for every superintuitionistic calculus, i.e., a finitely axiomatizable extension of the intuitionistic propositional calculus. Particularly, this holds for the classical propositional calculus and the Linial and Post theorem.

In 1961, A. A. Markov (Jr.) proposed the following problem: is it decidable whether a given finite set of implicational propositional formulas, i.e., formulas over the signature $\{\to\}$, axiomatizes all classical implicational tautologies? Kuznetsov in~\cite{Kuznetsov:63} mentioned that this problem seems to be still open.

In 1994, Marcinkowski~\cite{Marcinkowski:94} proved that Markov's problem is undecidable. Moreover, Marcinkowski obtained a much stronger result: fix an implicational propositional tautology $A$ that is not of the form $B \to B$ for some formula $B$, then it is undecidable whether $A$ is derivable from a given finite set of implicational formulas by the rules of modus ponens and substitution.

Recently, Zolin in 2013~\cite{Zolin:2013} re-established the result of Kuznetsov for the superintuitionistic propositional calculus over the signatures $\{ \wedge, \to \}$ and $\{ \vee, \to \}$. It is based on the so-called tag systems introduced by Post~\cite{Post:43} and proposed in 2009 by Bokov~\cite{Bokov:2009} for the proof of the result of Linial and Post. Besides Zolin in~\cite{Zolin:2013} gave a detailed and useful historical survey of related results.

The aim of this paper is to prove the undecidability of the problem of recognizing axiomatizations for every superintuitionistic implicational propositional calculus over a signature containing the connective $\to$. By a superintuitionistic implicational propositional calculus we mean a finitely axiomatizable extension of intuitionistic implicational propositional calculus.

The paper is organized as follows. In the first part we introduce the basic notation, give a historical survey of related results, and state our main result. In the second part we reduce the halting problem for tag systems to the problem of recognizing axiomatizations for propositional calculi, and prove our main result.

\section{Preliminaries and results}

Let $\mathcal{V}$ be an infinite set of propositional variables. Letters $x, y, z, u$, etc., are used to denote propositional variables. The signature $\Sigma$ is a finite set of connectives. Each connective is associated with a unique, classical, two-valued truth-function. Usually connectives are binary or unary such as $\{\neg, \vee, \wedge, \to\}$.

Propositional formulas or $\Sigma$-formulas are built up from the signature $\Sigma$ and propositional variables from $\mathcal{V}$ in the usual way. Capital letters $A, B, C$, etc., are used to denote propositional formulas. Throughout the paper, we will omit the outermost parentheses in formulas and parentheses assuming the customary priority of connectives.

In this paper, we will consider arbitrary signatures containing the binary connective $\to$. Note that by Gladstone~\cite{Gladstone:65} we can suppose that the signature $\Sigma$ does not contain the symbol~$\to$, but there is some propositional formula having $x, y$ as sole variables, whose truth-table interpretation is ``$x$ implies $y$''. In this case we denote the specified formula simply by $x \to y$.

A \emph{propositional calculus} $P$ over a signature $\Sigma$ (or a $\Sigma$-\emph{calculus}) is a system consisting of a finite set $P$ of $\Sigma$-formulas referred to as \emph{axioms} and two rules of inference:

1) \emph{modus ponens}
\begin{equation*}
  A, A \to B \vdash B;
\end{equation*}

2) \emph{substitution}
\begin{equation*}
  A \vdash \sigma A,
\end{equation*}
where $\sigma A$ is the substitution instance of $A$, i.e., the result of applying the substitution $\sigma$ to the formula $A$.

Denote by $[P]$ the set of derivable (or provable) formulas of a calculus $P$. A \emph{derivation} in $P$ is defined from the axioms and the rules of inference in the usual way. The statement that a formula $A$ is drivable from $P$ is denoted by $P \vdash A$.

Let us introduce the following partial order relation on the set of all propositional calculus. We write $P_1 \leq P_2$ (or, equivalently, $P_2 \geq P_1$) if each drivable formula of $P_1$ is also drivable from $P_2$, i.e., if $[P_1] \subseteq [P_2]$. We write $P_1 \sim P_2$ and say that two calculi $P_1$ and $P_2$ are \emph{equivalent} if $[P_1] = [P_2]$. Finally, we write $P_1 < P_2$ if $[P_1] \subsetneq [P_2]$.

Denote by $\mathbf{Cl}_{\Sigma}$ the classical propositional calculus over a signature $\Sigma$, and by $\mathbf{Int}_{\Sigma}$ the intuitionistic propositional calculus over a signature $\Sigma$~\cite{Kleene:2002}. We assume that the signature $\Sigma$ of the intuitionistic propositional calculus $\mathbf{Int}_{\Sigma}$ is a subset of the following set of connectives $\{\wedge, \vee, \neg,  \to, \leftrightarrow, \top, \bot\}$.

Consider the intuitionistic implicational propositional calculus $\mathbf{Int}_{\{\to\}}$ with the set of axioms~\cite[p.69]{HilbertBernays:68}:

\smallskip \noindent
\begin{center}
\begin{tabular}{ll}
  $(\mathrm{A}_1)$ & $x \to (y \to x)$,  \\
  $(\mathrm{A}_2)$ & $(x \to (y \to z)) \to ((x \to y) \to (x \to z))$. \\
\end{tabular}
\end{center}
\smallskip

\noindent The classical implicational propositional calculus $\mathbf{Cl}_{\{\to\}}$ is obtained from $\mathbf{Int}_{\{\to\}}$ by adding the Peirce law $((x \to y) \to x) \to x$~\cite[p.52]{Tarski:83}.

Now we define some recognizing problems for a fixed propositional calculus $P_0$.

\begin{Problem}[Recognizing axiomatizations]
Given a propositional calculus $P$, determine whether $P_0 \sim P$.
\end{Problem}

\begin{Problem}[Recognizing extensions]
Given a propositional calculus $P$, determine whether $P_0 \leq P$.
\end{Problem}

\begin{Problem}[Recognizing completeness]
Given a propositional calculus $P$ such that $P \leq P_0$, determine whether $P_0 \leq P$.
\end{Problem}

The previous results can be summarized as follows.

\begin{Theorem}[Linial and Post, 1949]
The problems of recognizing axiomatizations, extensions, and completeness for $\mathbf{Cl}_{\{\neg, \vee\}}$ are undecidable.
\end{Theorem}

\begin{Theorem}[Kuznetsov, 1963]
Fix a calculus $P_0 \geq \mathbf{Int}_{\{\neg, \vee, \&, \to\}}$, then the problems of recognizing axiomatizations, extensions, and completeness for $P_0$ are undecidable.
\end{Theorem}

\begin{Theorem}[Marcinkowski, 1994] \label{T:Marcinkowski}
Fix a $\{\to\}$-tautology $A$ that is not of the form $B \to B$ for some formula $B$, then the problem of recognizing extensions for the $\{\to\}$-calculus $\{A\}$ is undecidable.
\end{Theorem}

Since the implicational calculi $\mathbf{Cl}_{\{\to\}}$ and $\mathbf{Int}_{\{\to\}}$ can be axiomatized by the following single formulas, as shown by  {\L}ukasiewicz~\cite{Lukasiewicz:48} and Meredith~\cite{Meredith:53},
\begin{align*}
  \mathbf{Cl}_{\{\to\}}  & \sim \{ ((x \to y) \to z) \to ((z \to x) \to (u \to x)) \} \\
  \mathbf{Int}_{\{\to\}} & \sim \{ ((x \to y) \to z) \to (u \to ((y \to (z \to v)) \to (y \to v))) \}
\end{align*}
the following result also makes sense.

\begin{Corollary}
The problems of recognizing axiomatizations, extensions, \\ and completeness for $\mathbf{Cl}_{\{\to\}}$ and the problem of recognizing extensions for $\mathbf{Int}_{\{\to\}}$ are undecidable.
\end{Corollary}

In 1930, Tarski~\cite{Tarski:83} proved that every propositional calculus, which contains the formulas $x \to ( y \to x )$ and $x \to ( y \to ( ( x \to ( y \to z ) ) \to z ) )$, can be axiomatized by a single formula. Since these formulas are derivable from $\mathbf{Int}_{\{\to\}}$, we have the following corollary of the Marcinkowski result.

\begin{Corollary}
\end{Corollary}
Fix a signature $\Sigma \supseteq \{\to\}$ and a $\Sigma$-calculus $P_0 \geq \mathbf{Int}_{\{\to\}}$, then the problem of recognizing extensions for $P_0$ is undecidable.

\begin{Theorem}[Zolin, 2013]
Fix a signature $\Sigma \supseteq \{\wedge, \to\}$ and a $\Sigma$-calculus $P_0 \geq \mathbf{Int}_{\{\wedge, \to\}}$, then the problems of recognizing axiomatizations, extensions, and completeness for $P_0$ are undecidable.
\end{Theorem}

Our main result is the following theorems.

\begin{Theorem} \label{T:main}
Fix a signature $\Sigma \supseteq \{\to\}$ and a $\Sigma$-calculus $P_0 \geq \mathbf{Int}_{\{\to\}}$, then the problems of recognizing axiomatizations and completeness for $P_0$ are undecidable.
\end{Theorem}

\section{The proof of undecidability}

In order to prove Theorem~\ref{T:main}, we shall effectively reduce the halting problem for tag systems to the problem of recognizing completeness for propositional calculi. Then, the proof of Theorem~\ref{T:main} is immediate from the undecidability of the halting problem~\cite{Minsky:61}.

More precisely, we fix any signature $\Sigma$ such that $\{ \to \} \subseteq \Sigma$ and any $\Sigma$-calculus $P_0 \geq \mathbf{Int}_{\{\to\}}$. For a given tag system $T$ and a word $\omega$, we will construct a $\Sigma$-calculus $P = P_{T,\omega,P_0}$ such that $P \leq P_0$ and $T$ halts on the input word $\omega$ iff $P_0 \leq P$.

First let us recall the notion of a tag system introduced by Post~\cite{Post:43}.

\subsection{Tag systems}

Let $\mathcal{A}$ be a finite alphabet of letters $a_1, \dots, a_m$. By $\mathcal{A}^*$ denote the set of all words over $\mathcal{A}$, including the empty word. For $\alpha \in \mathcal{A}^*$, denote by $|\alpha|$ the length of the word $\alpha$.

\begin{Definition}[Post,~\cite{Post:43}]
A \emph{tag system} is a triple $T = \langle \mathcal{A}, \mathcal{W}, d \rangle$, where $\mathcal{A} = \{ a_1, \dots, a_m \}$ is a finite alphabet of $m$ symbols, $\mathcal{W} = \{ \omega_1, \dots, \omega_m\} \subseteq \mathcal{A}^*$ is a set of $m$ words, and $d \in \mathbb{N}$ is a \emph{deletion number}. Each words $\omega_i$ is associated to the letters $a_i$: $a_1 \to \omega_1, \dots, a_m \to \omega_m$.
\end{Definition}

We say that $T$ is applicable to a word $\alpha \in \mathcal{A}^*$ if $|\alpha| \geq d$. The application of $T$ to a word $\alpha \in \mathcal{A}^*$ is defined as follows. Examine the first letter of the word $\alpha$. If it is $a_i$ then
\begin{enumerate}
  \item remove the first $d$ letters from $\alpha$, and
  \item append to its end the word $\omega_i$.
\end{enumerate}
Perform the same operation on the resulting word, and repeat the process so long as the resulting word has $d$ or more letters. To be precise, if $\alpha = a_i \beta \gamma$, $|\beta| = d-1$, and $\gamma \in \mathcal{A}^*$, then $T$ produces the word $\gamma \omega_i$ from the word $a_i \beta \gamma$. Denote this production by $a_i \beta \gamma \stackrel{T}{\longmapsto} \gamma \omega_i$. We write $\alpha \stackrel{T}{\Longmapsto} \beta$ if there are words $\gamma_1, \dots, \gamma_n$, $n \geq 1$, such that $\alpha = \gamma_1$, $\beta = \gamma_n$, and $\gamma_i \stackrel{T}{\longmapsto} \gamma_{i+1}$ for all $1 \leq i \leq n-1$.

Define the halting problem of tag systems. We say that a tag system $T$ \emph{halts} on a word $\alpha \in \mathcal{A}^*$ if there exists a word $\beta \in \mathcal{A}^*$ such that $\alpha \stackrel{T}{\Longmapsto} \beta$ and $T$ is not applicable to $\beta$, i.e. $|\beta| < d$. The \emph{halting problem} for a fixed tag system $T$ is, given any word $\alpha \in \mathcal{A}^*$, to determine whether $T$ halts on $\alpha$.

\begin{Theorem}[Minsky,~\cite{Minsky:61}] \label{T:Minsky}
There is a tag system $T$ for which the halting problem is undecidable.
\end{Theorem}

Moreover, Wang~\cite{Wang:63} showed that this holds even for some tag system $T$ with $d = 2$ and $1 \leq |\omega_i| \leq 3$ for all $1 \leq i \leq m$. For this reason, throughout the paper we will assume that all words $\omega_i$ are nonempty.

\subsection{Encoding of letters and words}

Let $\mathcal{A}$ be a finite set $\{a_1, \dots, a_m\}$. The set of all nonempty words over $\mathcal{A}$ is denoted by $\mathcal{A}^+$. We encode letters and words on $\mathcal{A}$ as $\{\to\}$-formulas.

Fix a variable $x^0$ not occurring in $P_0$. Then the code of the letter $a_i \in \mathcal{A}$, for $1 \leq i \leq m$, is a formula
\begin{equation*}
  \overline{a_i} := ( (x^0 \to \underbrace{x^0 ) \to \dots \to x^0 )}_i \to ( x^0 \to ( x^0 \to x^0 ) ).
\end{equation*}
It is easily shown that $\mathbf{Int}_{\{\to\}} \vdash B \to A$ whenever $\mathbf{Int}_{\{\to\}} \vdash A$. Since $x^0 \to ( x^0 \to x^0 )$ is a substitution instance of the axiom $\mathrm{A}_1$, we have the following lemma.

\begin{Lemma} \label{L:DerivabilityOfCodeOfLetter}
$\mathbf{Int}_{\{\to\}} \vdash \overline{a}$, for every letter $a \in \mathcal{A}$.
\end{Lemma}

Now we introduce the following notation. Let $x \vee y$ be an abbreviation for the following formula:
\begin{equation*}
  (x \to y) \to y.
\end{equation*}
For a word $\alpha = a_{i_1} \dots a_{i_k} \in \mathcal{A}^+$, we write $\overrightarrow{\alpha}$ as a shortcut for the formula
\begin{equation*}
  \overline{a_{i_1}} \vee \left( \overline{a_{i_2}} \vee \dots \vee \left(\overline{a_{i_{k-1}}} \vee \overline{a_{i_k}} \right)\right),
\end{equation*}
and $\overleftarrow{\alpha}$ as a shortcut for the formula
\begin{equation*}
  \left( \left ( \overline{a_{i_1}} \vee \overline{a_{i_2}} \right) \vee \dots \vee \overline{a_{i_{k-1}}} \right) \vee \overline{a_{i_k}}.
\end{equation*}
The notation can be extended to the alphabet $\mathcal{A} \cup \mathcal{V}$, where $\mathcal{V}$ is the infinite set of propositional variables defined above. For example, $\overrightarrow{a x b y} = \overline{a} \vee \left(x \vee \left(\overline{b} \vee y\right)\right)$, where $a, b \in \mathcal{A}$ and $x, y \in \mathcal{V}$.

\begin{Lemma} \label{L:DerivabilityOfDisjunction}
In $\mathbf{Int}_{\{\to\}}$ the following derivations hold:
\begin{align*}
  \mathbf{Int}_{\{\to\}} & \vdash x \to x \vee y, \\
  \mathbf{Int}_{\{\to\}} & \vdash y \to x \vee y.
\end{align*}
\end{Lemma}
\begin{proof}
The formula $y \to x \vee y$ is the substitution instance of the axiom $\mathrm{A}_1$. Since
\begin{equation*}
  x, x \to y \vdash y,
\end{equation*}
we have $\mathbf{Int}_{\{\to\}} \vdash x \to x \vee y$ by the deduction theorem.
\end{proof}

\begin{Definition} (Zolin,~\cite{Zolin:2013}) \label{D:AlphabeticFormula}
An \emph{alphabetic formula} over the alphabet $\mathcal{A}$, or an $\mathcal{A}$-\emph{formula} for short, is an arbitrary $\{\vee\}$-formula over the codes of letters from $\mathcal{A}$. Formally, $\overline{a}$ is a $\mathcal{A}$-formula for each letter $a \in \mathcal{A}$, and if $A$, $B$ are $\mathcal{A}$-formulas then so is $A \vee B$.
\end{Definition}

In particular, $\overrightarrow{\alpha}$ and $\overleftarrow{\alpha}$ are $\mathcal{A}$-formulas. Lemma~\ref{L:DerivabilityOfCodeOfLetter} and Lemma~\ref{L:DerivabilityOfDisjunction} imply:

\begin{Lemma} \label{L:DerivabilityOfAFormulas}
$\mathbf{Int}_{\{\to\}} \vdash A$, for every $\mathcal{A}$-formula $A$.
\end{Lemma}

Given a formula $A$, denote by $A^*$ the set of all substitution instances of $A$. Similarly, given a set $M$ of formulas, denote by $M^*$ the set
\begin{equation*}
  M^* := \bigcup_{A \in M} A^*.
\end{equation*}
In accordance with~\cite{Zolin:2013} let us call two formulas $A$ and $B$ \emph{unifiable} if $A^* \cap B^* \neq \emptyset$.

\begin{Lemma} \label{L:AlphabeticFormulas}
No two distinct $\mathcal{A}$-formulas are unifiable.
\end{Lemma}
\begin{proof}
By induction on the definition of an $\mathcal{A}$-formula $A$.

Let $A$ be the code of a letter $a_i \in \mathcal{A}$. If $B$ is the code of a letter $a_j \in \mathcal{A}$, then $i \neq j$. Without loss of generality, $i < j$. Denote by $C$ the following formula
\begin{equation*}
  ((y \to \underbrace{x^0) \to \dots \to x^0)}_i \to (x^0 \to (x^0 \to x^0)).
\end{equation*}
Since $\overline{a_i}$ is the substitution instance of $C$ with respect to replacing the propositional variable $y$ by $x^0$ and $\overline{a_j}$ is the substitution instance of $C$ with respect to replacing the propositional variable $y$ by
\begin{equation*}
  ((x^0 \to \underbrace{x^0) \to \dots \to x^0)}_{j-i},
\end{equation*}
we conclude that $A$ and $B$ are not unifiable.

If $B$ is a formula $B_1 \vee B_2$ for some $\mathcal{A}$-formulas $B_1$ and $B_2$, then $A = \overline{a_i}$ is a substitution instance of
\begin{equation*}
(y \to x^0) \to (x^0 \to (x^0 \to x^0))
\end{equation*}
and $B$ is the substitution instance of $(u \to v) \to v$. Since the formulas $x^0$ and $x^0 \to (x^0 \to x^0)$ are not unifiable, we see that $A$ and $B$ are not unifiable either.

Now let $A = A_1 \vee A_2$ for some $\mathcal{A}$-formulas $A_1$ and $A_2$, so it can be assumed that $B = B_1 \vee B_2$ for some $\mathcal{A}$-formulas $B_1$ and $B_2$. If $A$, $B$ are unifiable, then also $A_1$, $B_1$ and $A_2$, $B_2$ are unifiable. By induction hypothesis, $A_1 = B_1$ and $A_2 = B_2$. Hence, $A = B$.

This completes the proof of the lemma.
\end{proof}

Denote by $\rhd$ the following formula
\begin{equation*}
((x^0 \to x^0) \to x^0) \to x^0.
\end{equation*}
Since formulas $x \to x$ and $(y \to z) \to z$ are not unifiable, we obtain the following lemma.

\begin{Lemma} \label{L:Triangle}
Formulas $\rhd$, $\rhd \to A$ are not unifiable for any formula $A$, and formulas $\rhd$, $(\rhd \to B) \to C$ are also not unifiable for any formulas $B$, $C$.
\end{Lemma}

Next, we define the code of a word $\alpha \in \mathcal{A}^+$ as the finite set $\mathsf{Code}(\alpha)$ consisting of all tautologies of the following four types:

\bigskip

\begin{tabular}{lll}
  Type 0 & $\rhd \to \overrightarrow{\alpha}$ &  \\
  Type 1 & $\rhd \to \overrightarrow{\alpha_1} \vee \overrightarrow{\alpha_2}$ & $\alpha = \alpha_1 \alpha_2$, $|\alpha_1| \geq 2$, $|\alpha_2| \geq 1$; \\
  Type 2 & $\rhd \to (\overleftarrow{\alpha_1} \vee \overrightarrow{\alpha_2}) \vee \overrightarrow{\alpha_3}$ & $\alpha = \alpha_1 \alpha_2 \alpha_3$, $|\alpha_1| \geq 2$, $|\alpha_2| \geq 2$, $|\alpha_3| \geq 1$; \\
  Type 3 & $\rhd \to \overleftarrow{\alpha_1} \vee \overrightarrow{\alpha_2}$ & $\alpha = \alpha_1 \alpha_2$, $|\alpha_1| \geq 3$, $|\alpha_2| \geq 1$. \\
\end{tabular}

\bigskip

\noindent Furthermore, we will call each formula of $\mathsf{Code}(\alpha)$ as the code of same word $\alpha$. The code of type $0$ is said to be \emph{canonical}.

\subsection{Construction of the calculus $P_{T,\omega,P_0}$}

Let $T = \langle \mathcal{A}, \mathcal{W}, d \rangle$ be a tag system, $\omega$ a nonempty word over $\mathcal{A}$, and $P_0$ a $\Sigma$-calculus. Recall that $\mathcal{A} = \{ a_1, \dots, a_m \}$, $\mathcal{W} = \{ \omega_1, \dots, \omega_m\}$, and all $\omega_i$ are assumed to be nonempty. Denote by $P_{T,\omega,P_0}$ a $\Sigma$-calculus with axioms:

\bigskip

\noindent
\begin{tabular}{llll}
  $(\mathrm{W}_{\omega})$   & $\rhd \to \overrightarrow{\omega}$, &  \\
  $(\mathrm{T}_1)$ & $(\rhd \to \overrightarrow{a_i \alpha y\ }) \to (\rhd \to \overrightarrow{y \omega_i})$, & \multicolumn{2}{l}{for all $\alpha \in \mathcal{A}^*$, $|\alpha| = d-1$, $1 \leq i \leq m$,} \\
  $(\mathrm{T}_2)$ & $(\rhd \to \overrightarrow{a_i \alpha\ }) \to (\rhd \to \overrightarrow{\omega_i})$, & \multicolumn{2}{l}{for all $\alpha \in \mathcal{A}^*$, $|\alpha| = d-1$, $1 \leq i \leq m$,} \\
  $(\mathrm{H})$ & $(\rhd \to \overrightarrow{\alpha}) \to A$, & \multicolumn{2}{l}{for all $\alpha \in \mathcal{A}^*$, $0 < |\alpha| < d$, $A \in P_0$,} \\
  $(\mathrm{R}_1)$ & \multicolumn{2}{l}{$(\rhd \to (y \vee \overrightarrow{az}) \vee u) \to (\rhd \to (\overleftarrow{ya} \vee z) \vee u)$,} & for all $a \in \mathcal{A}$, \\
  $(\mathrm{R}_2)$ & \multicolumn{2}{l}{$(\rhd \to \overleftarrow{ya} \vee z) \to (\rhd \to y \vee \overrightarrow{az})$,} & for all $a \in \mathcal{A}$. \\
\end{tabular}

\bigskip

Let $P_{T}$ be the subsystem of $P_{T,\omega,P_0}$ consisting of axioms $\mathrm{T}_1$, $\mathrm{T}_2$, $\mathrm{R}_1$, $\mathrm{R}_2$ and $P_{T, \omega} = P_T \cup \{\mathrm{W}_{\omega}\}$. Now we prove some properties of the calculus $P_{T,\omega,P_0}$.

\begin{Lemma}
$P_{T,\omega} \leq \mathbf{Int}_{\{\to\}}$.
\end{Lemma}
\begin{proof}
Easily follows from Lemmas~\ref{L:DerivabilityOfCodeOfLetter},~\ref{L:DerivabilityOfDisjunction} and~\ref{L:DerivabilityOfAFormulas}.
\end{proof}

\begin{Corollary} \label{C:Inclusion}
$P_{T,\omega,P_0} \leq P_0$.
\end{Corollary}

\subsection{Derivability of the $T$-productions}

Here we show that the calculus $P_T$ can ``simulate'' productions of the tag system $T$. At the beginning let us prove auxiliary lemmas.

\begin{Lemma}
$\mathrm{R}_1, \rhd \to (\overleftarrow{\xi} \vee \overrightarrow{\beta}) \vee \overrightarrow{\zeta} \vdash \rhd \to \overleftarrow{\xi \beta} \vee \overrightarrow{\zeta}$, for all $\xi, \beta, \zeta \in \mathcal{A}^+$.
\end{Lemma}
\begin{proof}
By induction on $|\beta|$. If $|\beta| = 1$, then the formulas $\rhd \to (\overleftarrow{\xi} \vee \overrightarrow{\beta}) \vee \overrightarrow{\zeta}$ and $\rhd \to \overleftarrow{\xi \beta} \vee \overrightarrow{\zeta}$ are identical.

Now let $|\beta| \geq 2$, then $\beta = a \delta$ for a letter $a \in \mathcal{A}$ and a nonempty word $\delta$. Therefore,
\begin{equation*}
  \mathrm{R}_1,~\rhd \to (\overleftarrow{\xi} \vee \overrightarrow{a \delta}) \vee \overrightarrow{\zeta}~\vdash~\rhd \to (\overleftarrow{\xi a} \vee \overrightarrow{\delta}) \vee \overrightarrow{\zeta}
\end{equation*}
by modus ponens. By induction hypothesis, we have
\begin{equation*}
  \mathrm{R}_1,~\rhd \to (\overleftarrow{\xi a} \vee \overrightarrow{\delta}) \vee \overrightarrow{\zeta}~\vdash~\rhd \to \overleftarrow{\xi \beta} \vee \overrightarrow{\zeta}.
\end{equation*}
This completes the proof of the lemma.
\end{proof}

\begin{Corollary}
$\mathrm{R}_1,~\rhd \to \overrightarrow{\xi} \vee \overrightarrow{\zeta}~\vdash~\rhd \to \overleftarrow{\xi} \vee \overrightarrow{\zeta}$, for all $\xi, \zeta \in \mathcal{A}^+$.
\end{Corollary}

\begin{Lemma}
$\mathrm{R}_2,~\rhd \to \overleftarrow{\xi} \vee \overrightarrow{\zeta}~\vdash~\rhd \to \overrightarrow{\xi\zeta}$, for all $\xi, \zeta \in \mathcal{A}^+$.
\end{Lemma}
\begin{proof}
By induction on $|\xi|$. If $|\xi| = 1$, then the formulas $\rhd \to \overleftarrow{\xi} \vee \overrightarrow{\zeta}$ and $\rhd \to \overrightarrow{\xi\zeta}$ are identical.

Now let $|\xi| \geq 2$, then $\xi = \beta a$ for a letter $a \in \mathcal{A}$ and a nonempty word $\beta$. Therefore,
\begin{equation*}
  \mathrm{R}_2,~\rhd \to \overleftarrow{\beta a} \vee \overrightarrow{\zeta}~\vdash~\rhd \to \overleftarrow{\beta} \vee \overrightarrow{a \zeta}
\end{equation*}
by modus ponens. By induction hypothesis, we have
\begin{equation*}
  \mathrm{R}_2,~\rhd \to \overleftarrow{\beta} \vee \overrightarrow{a \zeta}~\vdash~\rhd \to \overrightarrow{\xi\zeta}.
\end{equation*}
This completes the proof of the lemma.
\end{proof}

\begin{Corollary} \label{C:CodeUniforme}
$\mathrm{R}_1,~\mathrm{R}_2,~\rhd \to \overrightarrow{\xi} \vee \overrightarrow{\zeta}~\vdash~\rhd \to \overrightarrow{\xi\zeta}$, for all $\xi, \zeta \in \mathcal{A}^+$.
\end{Corollary}

\begin{Lemma} \label{L:Derivability}
If $\xi \stackrel{T}{\longmapsto} \zeta$ then $P_T,~\rhd \to \overrightarrow{\xi}~\vdash~\rhd \to \overrightarrow{\zeta}$, for all $\xi, \zeta \in \mathcal{A}^+$.
\end{Lemma}
\begin{proof}
Since $T$ is applicable to $\xi$, we have $|\xi| \geq d$. Therefore, $\xi = a_i \alpha \beta$ and $\zeta = \beta \omega_i$, where $|\alpha| = d-1$ and $|\beta| \geq 0$.

\noindent If $|\beta| = 0$, then

\bigskip

\begin{tabular}{rcll}
  $P_T$ & $\vdash$ & $(\rhd \to \overrightarrow{\xi}) \to (\rhd \to \overrightarrow{\zeta})$ & by the axiom $(\mathrm{T}_2)$, and \\
  $P_T,~\rhd \to \overrightarrow{\xi}$ & $\vdash$ & $\rhd \to \overrightarrow{\zeta}$ & by modus ponens. \\
\end{tabular}

\bigskip

\noindent Let $|\beta| > 0$, so

\bigskip

\begin{tabular}{rcll}
  $P_T$ & $\vdash$ & $(\rhd \to \overrightarrow{\xi}) \to (\rhd \to \overrightarrow{\beta} \vee \overrightarrow{\omega_i})$ & by the axiom $(\mathrm{T}_1)$, \\
  $P_T,~\rhd \to \overrightarrow{\xi}$ & $\vdash$ & $\rhd \to \overrightarrow{\beta} \vee \overrightarrow{\omega_i}$ & by modus ponens, \\
  $P_T,~\rhd \to \overrightarrow{\xi}$ & $\vdash$ & $\rhd \to \overrightarrow{\zeta}$ & by Corollary~\ref{C:CodeUniforme}. \\
\end{tabular}

\bigskip

\noindent The lemma is proved.
\end{proof}

\begin{Corollary} \label{C:Derivability}
If $\xi \stackrel{T}{\Longmapsto} \zeta$ then $P_T,~\rhd \to \overrightarrow{\xi}~\vdash~\rhd \to \overrightarrow{\zeta}$, for all $\xi, \zeta \in \mathcal{A}^+$.
\end{Corollary}

The proof is trivial by definition of the tag system.

\subsection{Production of the $P_T$-derivations}

Here we show that the tag system $T$ can produce, on the input word $\omega$, the words whose codes have derivations in $P_{T, \omega, P_0}$ of a ``small'' height (to be defined below). As a preliminary let us introduce some notation and prove auxiliary lemmas.

Given $\alpha \in \mathcal{A}^*$, denote by $\mathsf{Code}_T(\alpha)$ the set of formulas:
\begin{equation*}
  \mathsf{Code}_T(\alpha) := \bigcup_{\beta \in \mathcal{A}^*,~\alpha \stackrel{T}{\Longmapsto} \beta} \mathsf{Code}(\beta).
\end{equation*}
It is clear that $\mathsf{Code}(\alpha) \subseteq \mathsf{Code}_T(\alpha)$ for all $\alpha \in \mathcal{A}^*$.

For any propositional calculus $P$, denote by $\left\langle P \right\rangle$ the set of propositional formulas obtained from $P$ by applying modus ponens and substitution once:
\begin{align*}
  \left\langle P \right\rangle := & \left\{ B \mid A, A \to B \in P \text{ for some formula } A \right\} \cup \\
  & \left\{ \sigma A \mid A \in P \text{ and } \sigma \text{ is a substitution} \right\}.
\end{align*}
Furthermore, let $\left\langle P \right\rangle_0 = P$ and
\begin{equation*}
  \left\langle P \right\rangle_{n+1} = \left\langle \left\langle P \right\rangle_n \right\rangle
\end{equation*}
for $n \geq 0$. It follows easily that $\left\langle P \right\rangle_n \subseteq \left\langle P \right\rangle_{n+1}$ for all $n \geq 0$ and the set $[P]$ of all derivable formulas of the calculus $P$ can be represented as
\begin{equation*}
  [P] = \left\langle P \right\rangle_{\infty} = \bigcup_{n \geq 0} \left\langle P \right\rangle_n.
\end{equation*}
Let $A$ be a formula derivable from $P$. We say that $A$ has the \emph{derivation height} $n$, if $A \in \left\langle P \right\rangle_n$ and $A \notin \left\langle P \right\rangle_{n-1}$.

Consider the tag system $T$ and the calculus $P_{T,\omega,P_0}$. Let $T$ halts on the input word $\omega$, we take the minimal $n \geq 0$ such that $\left\langle P_{T,\omega,P_0} \right\rangle_n$ contains at least one substitution instance of the code of some word $\alpha \in \mathcal{A}^*$ with $|\alpha| < d$:
\begin{equation*}
  N_{\omega} = \min \{ n \geq 0 \mid \mathsf{Code}^*(\alpha) \cap \left\langle P_{T, \omega, P_0} \right\rangle_{n} \neq \emptyset, \text{ for some } \alpha \in \mathcal{A}^* \text{ with } |\alpha| < d \}.
\end{equation*}
If $T$ does not halt, then we put $N_{\omega} = \infty$. Recall that $\mathsf{Code}^*(\alpha)$ is the set of all substitution instances of formulas in $\mathsf{Code}(\alpha)$. Denote $P_{T, P_0} = P_T \cup \{\mathrm{H}\}$.

\begin{Lemma} \label{L:FormOfDerivations}
$\left\langle P_{T,\omega,P_0} \right\rangle_{N_{\omega}} \subseteq \mathsf{Code}_T^*(\omega) \cup P_{T, P_0}^*$ for all $\omega \in \mathcal{A}^*$.
\end{Lemma}
\begin{proof}
We will prove by induction on $n \leq N_{\omega}$ that
\begin{equation*}
  \left\langle P_{T,\omega,P_0} \right\rangle_n \subseteq \mathsf{Code}_T^*(\omega) \cup P_{T, P_0}^*.
\end{equation*}

If $n = 0$, then $\left\langle P_{T,\omega,P_0} \right\rangle_0 = P_{T,\omega,P_0}$. It can easily be checked that the axiom $\mathrm{W}_{\omega}$ is in $\mathsf{Code}_T^*(\omega)$ and all the other axioms of $P_{T,\omega,P_0}$ are in $P_{T, P_0}^*$.

Let the induction assumption be satisfied for some $1 \leq n < N_{\omega}$. Since the right-hand side of the inclusion is closed under substitution, we only consider the case of a formula $B$ obtained by modus ponens from some formulas $A,\ A \to B \in \left\langle P_{T,\omega,P_0} \right\rangle_n$. By induction hypothesis,
\begin{equation*}
  \left\langle P_{T,\omega,P_0} \right\rangle_n \subseteq \mathsf{Code}_T^*(\omega) \cup P_{T, P_0}^*.
\end{equation*}
We claim that (1) $A \to B \in P_{T, P_0}^*$, and (2) $A \in \mathsf{Code}_T^*(\omega)$. Proofs are below. Then we will show that $B \in \mathsf{Code}_T^*(\omega)$, which suffices for proving Lemma~\ref{L:FormOfDerivations}.

Note that $\mathsf{Code}_T^*(\omega) \cap P_{T, P_0}^* = \emptyset$, due to Lemma~\ref{L:Triangle}.

\textbf{Proof of (1):} Assume the contrary: $A \to B \in \mathsf{Code}_T^*(\omega)$. Since $\rhd$ is the premise of any code of any word, we have $A \in \rhd^*$. However, $A \in \mathsf{Code}_T^*(\omega) \cup P_{T, P_0}^*$, which is impossible, because all formulas in $\mathsf{Code}_T^*(\omega)$ and $P_{T, P_0}^*$ have the form $(\rhd \to C)$ or $(\rhd \to C) \to D$, for some $C$, $D$, and so are not unifiable with $\rhd$ by Lemma~\ref{L:Triangle}.

\textbf{Proof of (2):} Assume the contrary: $A \in P_{T, P_0}^*$. Then $A$ is a substitution instance of a formula of the from $(\rhd \to C) \to D$, for some $C$, $D$. By (1), $A \to B \in P_{T, P_0}^*$. So, $A \to B$ is a substitution instance of a formula of the from $(\rhd \to E) \to F$, for some $E$ and $F$. This would apply that $(\rhd \to C)$ is unifiable with $\rhd$, which is impossible by Lemma~\ref{L:Triangle}.

We are going to show that $B \in \mathsf{Code}_T^*(\omega)$. Since $A \in \mathsf{Code}^*(\xi)$ for some word $\xi \in \mathcal{A}^+$ such that $\omega \stackrel{T}{\Longmapsto} \xi$, and $A \to B$ is a substitution instance of some of the 5 axioms in $P_{T, P_0}$, we need to consider the following 5 cases.

\textbf{Case 1.} $A \to B$ is a substitution instance of the axiom $\mathrm{T}_1$. Hence
\begin{equation*}
  A \in \left( \rhd \to \overrightarrow{a_i \alpha y} \right)^*
\end{equation*}
for some letter $a_i \in \mathcal{A}$ and a word $\alpha \in \mathcal{A}^*$ such that $|\alpha| = d-1$. Since the formula $A \in \mathsf{Code}^*(\xi)$, it is easily shown by Lemma~\ref{L:AlphabeticFormulas} that
\begin{equation*}
  A \in \left( \rhd \to \overrightarrow{a_i \alpha \gamma} \right)^*
\end{equation*}
for some $\gamma \in \mathcal{A}^+$, so that $\xi = a_i \alpha \gamma$. Therefore $B$ is the substitution instance of the code
\begin{equation*}
 \rhd \to \overrightarrow{\gamma} \vee \overrightarrow{\omega_i}
\end{equation*}
for the word $\zeta = \gamma \omega_i$ and $\xi \stackrel{T}{\longmapsto} \zeta$.

\textbf{Case 2.} $A \to B$ is a substitution instance of the axiom $\mathrm{T}_2$. Hence
\begin{equation*}
  A \in \left( \rhd \to \overrightarrow{a_i \alpha} \right)^*
\end{equation*}
for some letter $a_i \in \mathcal{A}$ and a word $\alpha \in \mathcal{A}^*$ such that $|\alpha| = d-1$. So, $\xi = a_i \alpha$. Therefore $B$ is the substitution instance of the code
\begin{equation*}
  \rhd \to \overrightarrow{\omega_i}
\end{equation*}
for the word $\zeta = \omega_i$ and $\xi \stackrel{T}{\longmapsto} \zeta$.

\textbf{Case 3.} $A \to B$ is a substitution instance of the axiom $\mathrm{H}$. This case is impossible, since otherwise we would have $A \in (\rhd \to \overrightarrow{\alpha})^*$ for some $\alpha \in \mathcal{A}^*$, $0 < |\alpha| < d$. This contradicts to the fact that
\begin{equation*}
  \left( \rhd \to \overrightarrow{\alpha} \right)^* \cap \left\langle P_{T, \omega, P_0} \right\rangle_{n} \neq \emptyset
\end{equation*}
and $n < N_{\omega}$.

\textbf{Case 4.} $A \to B$ is a substitution instance of the axiom $\mathrm{R}_1$. Hence
\begin{equation*}
  A \in \left( \rhd \to \left( y \vee \overrightarrow{az} \right) \vee u \right)^*
\end{equation*}
for some $a \in \mathcal{A}$. Since the formula $A \in \mathsf{Code}^*(\xi)$, we have by Lemma~\ref{L:AlphabeticFormulas} that
\begin{equation*}
  A \in \left( \rhd \to \left( \overleftarrow{\xi_1} \vee \overrightarrow{a\xi_2} \right) \vee \overrightarrow{\xi_3} \right)^*
\end{equation*}
for some $\xi_1, \xi_2, \xi_3 \in \mathcal{A}^+$ such that $\xi = \xi_1 a \xi_2 \xi_3$. Therefore $B$ is a substitution instance of the code
\begin{equation*}
  \rhd \to \left( \overleftarrow{\xi_1 a} \vee \overrightarrow{\xi_2} \right) \vee \overrightarrow{\xi_3}
\end{equation*}
for the same word $\xi = \xi_1 a \xi_2 \xi_3$.

\textbf{Case 5.} $A \to B$ is a substitution instance of the axiom $\mathrm{R}_2$. Hence
\begin{equation*}
  A \in \left( \rhd \to \overleftarrow{ya} \vee z \right)^*
\end{equation*}
for some $a \in \mathcal{A}.$ Since the formula $A \in \mathsf{Code}^*(\xi)$, we have by Lemma~\ref{L:AlphabeticFormulas} that
\begin{equation*}
  A \in \left( \rhd \to \overleftarrow{\xi_1 a} \vee \overrightarrow{\xi_2} \right)^*
\end{equation*}
for some $\xi_1, \xi_2 \in \mathcal{A}^+$ such that $\xi = \xi_1 a \xi_2$. Therefore $B$ is the substitution instance of the code
\begin{equation*}
  \rhd \to \overleftarrow{\xi_1} \vee \overrightarrow{a \xi_2}
\end{equation*}
for the same word $\xi = \xi_1 a \xi_2$.

Cases 1, 2, 3, 4, 5 exhaust all possibilities and so we have that $B \in \mathsf{Code}^*(\zeta)$ for some word $\zeta \in \mathcal{A}^*$ such that $\xi \stackrel{T}{\Longmapsto} \zeta$. Then $B \in \mathsf{Code}_T^*(\omega)$, since $\omega \stackrel{T}{\Longmapsto} \xi$ by induction hypothesis. The proof is completed.
\end{proof}

Now we prove that the code of each nonempty word over $\mathcal{A}$ derivable from $P_{T,\omega,P_0}$ with the derivation height less then or equal to $N_{\omega}$ is the code of a word produced from $\omega$ by the tag system $T$.

\begin{Corollary} \label{C:Production}
If $\mathsf{Code}^*(\alpha) \cap \left\langle P_{T,\omega,P_0} \right\rangle_{N_{\omega}} \neq \emptyset$ then $\omega \stackrel{T}{\Longmapsto} \alpha$, for all $\alpha \in \mathcal{A}^+$.
\end{Corollary}
\begin{proof}
By Lemma~\ref{L:FormOfDerivations}, we have
\begin{equation*}
  \left\langle P_{T,\omega,P_0} \right\rangle_{N_{\omega}} \subseteq \mathsf{Code}_T^*(\omega) \cup P_{T, P_0}^*.
\end{equation*}
Furthermore, the application of Lemma~\ref{L:Triangle} yields
\begin{equation*}
  \mathsf{Code}_T^*(\omega) \cap P_{T, P_0}^* = \emptyset.
\end{equation*}
It is obvious that $\mathsf{Code}^*(\alpha) \cap P_{T, P_0}^* = \emptyset$. Hence $\mathsf{Code}^*(\alpha) \cap \mathsf{Code}_T^*(\omega) \neq \emptyset$, and so $\omega \stackrel{T}{\Longmapsto} \alpha$ by definition of the set $\mathsf{Code}_T^*(\omega)$. The lemma is proved.
\end{proof}

\section{The proof of Theorem~\ref{T:main}}

Let us show that the following problem is undecidable: given a tag system $T$ and a word $\omega \in \mathcal{A}$, determine whether $P_0 \leq P_{T, \omega, P_0}$.

Indeed, if the tag system $T$ halts on the input word $\omega$, then $\omega \stackrel{T}{\Longmapsto} \alpha$ for some word $\alpha \in \mathcal{A}^+$ such that $|\alpha| < d$. Hence the code $\rhd \to \overrightarrow{\alpha}$ of $\alpha$ is derivable from $P_{T, \omega, P_0}$ by Corollary~\ref{C:Derivability}. If we recall that $P_{T, \omega, P_0}$ contains the formula
\begin{equation*}
  (\rhd \to \overrightarrow{\alpha}) \to A
\end{equation*}
for every $A \in P_0$, we obtain that $P_0 \leq P_{T, \omega, P_0}$.

Now assume $P_0 \leq P_{T, \omega, P_0}$. Since $\mathbf{Int}_{\{\to\}} \leq P_0$, so by Lemma~\ref{L:DerivabilityOfAFormulas}, we have
\begin{equation*}
  P_{T,\omega,P_0} \vdash \rhd \to \overrightarrow{\alpha}
\end{equation*}
for every $\alpha$ such that $|\alpha| < d$. Hence $N_{\omega} < \infty$. Fix any word $\alpha$ with $|\alpha| < d$ such that $\mathsf{Code}^*(\alpha) \cap \left\langle P_{T,\omega,P_0} \right\rangle_{N_{\omega}} \neq \emptyset$. By Corollary~\ref{C:Production}, we obtain $\omega \stackrel{T}{\Longmapsto} \alpha$. Therefore, $T$ halts on $\omega$.

Thus, we reduce the halting problem of tag systems to the problem of recognizing extensions for the $\Sigma$-calculus $P_0$. Since the halting problem of tag systems is undecidable by Theorem~\ref{T:Minsky} and $P_{T, \omega, P_0} \leq P_0$ by Corollary~\ref{C:Inclusion}, this completes the proof of undecidability of recognizing completeness. As corollary we have the undecidability of problem of recognizing axiomatizations.

\section{Acknowledgement}

The author is grateful to Evgeny Zolin for useful comments and advices that improved the manuscript.


\end{document}